\documentclass[12pt, a4paper]{amsart}
\usepackage{amsmath,amssymb,amscd,amsfonts}
\newtheorem{theorem}{Theorem}[section]
\newtheorem{rem}[theorem]{Remark}
\newtheorem{defn}[theorem]{Definition}

\newtheorem{corollary}[theorem]{Corollary}

\def\dbar{\overline\partial}
\def\ddbar{\partial\overline\partial}

\let\ol=\overline
\let\wt=\widetilde
\let\wh=\widehat
\def\bQ{{\mathbb Q}}

\def\bR{{\mathbb R}}
\def\bZ{{\mathbb Z}}
\def\bD{{\mathbb D}}

\begin{document}
\title[Relative]{
Relative adjoint transcendental classes\\
and Albanese maps of \\
compact K\"ahler manifolds with nef ricci curvature}

\author{Mihai P\u aun}

\address{Korea Institute for Advanced Study \\
School of Mathematics\\
85 Hoegiro, Dongdaemun-gu, Seoul 130-722, Korea.}
\email{paun@kias.re.kr}

\maketitle

\section{Introduction}

Let $p:X\to Y$ be a holomorphic surjective map, where $X$ and $Y$ are 
compact K\"ahler manifolds. We denote by $W\subset Y$ an
analytic set containing the singular values of $p$, and
let $X_0:= p^{-1}(Y\setminus W)$. 
Let $\{\beta\}\in H^{1,1}(X, \bR)$ 
be a real cohomology class of $(1,1)$-type, which contains a 
non-singular, semi-positive definite representative $\beta$. 

Our primary goal in this note is to 
investigate the positivity properties of the class
$$c_1(K_{X/Y})+ \{\beta\},$$
which are inherited from similar  \emph{fiberwise}
properties. 

\medskip

\noindent In this perspective, the main statement we obtain here is as follows.

\begin{theorem} Let $p: X\to Y$ be a surjective map. We consider  
a semi-positive class $\{\beta\}\in H^{1,1}(X, \bR)$, such that the adjoint class $c_1(K_{X_y})+ \{\beta\}|_{X_y}$
is K\"ahler for
any $y\in Y\setminus W$.
Then the relative 
adjoint class
$$c_1(K_{X/Y})+ \{\beta\}$$
contains a closed positive current $\Theta$, which equals a (non-singular)
semi-positive definite form on $X_0$.

\end{theorem}

\medskip
 
\noindent As a consequence of the proof of the previous result, the current $\Theta$ will be greater than a K\"ahler metric when restricted to any relatively compact open subset of $X_0$, provided that $\beta$ is a K\"ahler metric. Also, if $\beta\geq p^\star(\gamma)$ for some (1,1)-form $\gamma$ on $Y$, then we have
$$\Theta\geq p^\star(\gamma).$$

We remark that 
if the class $\beta$ is the first Chern class of a holomorphic $\bQ$-line bundle $L$,
that is to say, if
$$\{\beta\}\in H^{1,1}(X, \bR)\cap H^2(X, \bQ),$$
then there are many results concerning the positivity of the twisted relative 
canonical bundle, cf. \cite{bo3}, \cite{bp1}, \cite{bp3}, \cite{campana}, \cite{fujita}, \cite{griff}, \cite{hoer}, \cite{kaw1}, 
\cite{kaw2}, \cite{kaw3}, \cite{janos1}, \cite{janos2}, \cite{mou}, 
\cite{schumi1}, \cite{schumi2},
\cite{eckart1}, \cite{eckart2}, \cite{eckart3} to quote only a few.

The references \cite{schumi1}, \cite{siu1} are particularly important for us; indeed, a large part of the arguments 
presented by G. Schumacher in \cite{schumi1}, \cite{schumi2} will be used in our proof, as they rely on the complex Monge-Amp\`ere equation as \emph{substitute} for the
theory of linear bundles used in the other works quoted above
(see section 3.2 of this paper).

\medskip 

\noindent Before stating a few consequences of our main result,
we recall the following metric version of the usual notion of 
\emph{nef} line bundle in algebraic geometry, as it was introduced in \cite{demailly1}.

\begin{defn} Let $(X, \omega)$ be a compact complex manifold
endowed with a hermitian metric, and let 
$\{\rho\}$ be a real (1,1) class on $X$. We say that $\{\rho\}$
is nef (in metric sense) if for every $\varepsilon> 0$ there exists a function 
$f_\varepsilon\in {\mathcal C}^\infty(X)$ such that 
\begin{equation}\label{1}\rho+ \sqrt{-1}\ddbar f_\varepsilon\geq -\varepsilon \omega.
\end{equation}
\end{defn}

Thus the class $\{\rho\}$ is nef if it admits non-singular representatives with
arbitrary small negative part. It was established in \cite{demailly2} that if $X$ is projective and if
$\{\rho\}$ is the first Chern class of a line bundle $L$, then 
$L$ is nef in the algebro-geometric sense if and only if
$L$ is nef in metric sense.
\smallskip

\noindent Let ${\mathcal X}\to \bD$ be a non-singular K\"ahler family over the
unit disk. Then we have the following (direct) consequence of Theorem 1.1.

\begin{corollary} We assume that the bundle $\displaystyle K_{{\mathcal X}_t}$ is nef,
for any $t\in \bD$. Then $\displaystyle K_{{\mathcal X}/\bD}$ is nef.

\end{corollary}

\smallskip

\noindent We remark that in the context of the previous corollary, much more is 
expected to be true. For example, if the K\"ahler version of the 
\emph{invariance of plurigenera} turns out to be true, then it 
would be enough to assume in Corollary~1.3
that $\displaystyle K_{{\mathcal X}_0}$ is pseudo-effective 
in order to derive the conclusion that $\displaystyle K_{{\mathcal X}/\bD}$ is pseudo-effective.

\medskip
The second application of Theorem 1.1 concerns the 
Albanese morphism associated to a compact K\"ahler manifold $X$.
We denote by 
$q:= h^0(X, {\mathcal O}_X)$ the irregularity of $X$, and let 
$${\rm Alb}(X):= H^0(X, T^\star_X)^\star/H_1(X, \bZ)$$
be the Albanese torus of associated to $X$. We recall that the Albanese map
$\alpha_X: X\to {\rm Alb}(X)$ is defined as follows
$$\alpha_X(p)(\gamma):= \int_{p_0}^p\gamma$$
modulo the group $H_1(X, \bZ)$, i.e. modulo the integral of 
$\gamma$ along loops at $p_0$. 

We assume that $-K_X$ is nef, in the sense of the definition above.
It was conjectured by J.-P. Demailly, Th. Peternell and M. Schneider 
in \cite{demailly3} that $\alpha_X$ is surjective; some particular cases
of this problem are established in \cite{demailly3}, \cite{mp1},
\cite{capet}. If $X$ is assumed to be 
\emph{projective}, then the surjectivity of the Albanese map
was established by Q. Zhang in \cite{zhang}, by using in an essential manner
the \emph{char p} methods. More recently, in the article \cite{zhang1}, the same author
provides an alternative proof of this result, based on the semi-positivity of direct images.

\medskip

\noindent We settle here the conjecture in full generality.

\begin{theorem} Let $X$ be a compact K\"ahler manifold such that 
$-K_X$ is nef. Then its Albanese morphism $\alpha_X: X\to {\rm Alb}(X)$
is surjective.

\end{theorem}
\medskip

Besides Theorem 1.1, our proof is using some ideas from \cite{demailly3} and \cite{bouck}; we refer to the first 
paragraph for a detailed discussion about the connections with these articles.

\bigskip

\noindent Our paper is organized as follows. In the first 
paragraph we review the
proof of Theorem 1.4
under the additional assumption that $X$ is projective. As we have already mentioned, 
in this case the result is known, but the proof we will present is different from
the arguments in \cite{zhang}: actually, it can be seen as a 
simplified variation of some of the arguments invoked in
\cite{zhang1} (see also \cite{zhang2}). It is based upon a version of Theorem 1.1 under the hypothesis 
that the class $\{\beta\}$ corresponds to a line bundle (this result is completely covered by the 
literature on the subject, cf. \cite{bp3}, \cite{emanation1}). 

This serves us as a motivation for the
second paragraph, where we prove Theorem 1.1. In a word, we show that 
the so-called \emph{fiberwise twisted K\"ahler-Einstein} metric endows the bundle 
$K_{X/Y}|_{X_0}$ with a metric whose curvature is bounded 
from below by $-\beta$. 
Thus, the twisted version of the psh variation of the 
K\"ahler-Einstein metric established in \cite{schumi1} holds true. Finally, we show that
the local weights of the metric constructed in this way are bounded
near the analytic set $X\setminus X_0$. This is by no means 
automatic, given 
the tools which are involved in the proof (the approximation 
theorem in \cite{demailly4}, together with a precise 
version of the Ohsawa-Takegoshi extension theorem, \cite{bp3}).
The difficulty steams from the fact that in order to establish the estimates for the said
weights we cannot rely on the
geometry of the manifold $X_y$, 
as $y$ is approaching a singular value of the map $p$.

Finally, a complete proof of 
the Corollary 1.3 and Theorem 1.4 is provided, together with a few questions/comments.
\medskip

\noindent {\bf Acknowledgments.} 
It is my privilege and pleasure to thank S. Boucksom, F. Campana, J.-P. Demailly, A. H\"oring,
G. Schumacher, Y.-T. Siu, T. Peternell and H. Tsuji for enlightening 
discussions about many topics presented here.

\bigskip

\section {Surjectivity of the Albanese map: review of the projective case}

\medskip

\noindent As we have already mentioned in the introduction,
our proof for the subjectivity of the Albanese map corresponding to the compact K\"ahler
manifolds with nef anti-canonical class relies heavily on Theorem 1.1.
However, if the manifold $X$ is projective, then the following stronger version of Theorem 1.1 was obtained in \cite{bp3}, 
\cite{emanation1}.

\begin{theorem} Let $(F, h_F)$ be a line bundle on $X$, endowed with a metric with
semi-positive definite curvature form.
We assume moreover that for some generic $y\in Y$ the bundle
$$\displaystyle k(mK_{X_y}+ F)$$ 
admits a section which is $L^{2\over km}$--integrable with respect to 
$h_F^{1/m}$, where $m$ is a positive integer. 
Then the bundle $mK_{X/Y}+ F$ is pseudo-effective. 

\end{theorem}

\medskip

\noindent As a consequence, we infer the following statement.

\begin{corollary} Let $p: X\to Y$ be a surjective map between 
non-singular projective manifolds. We consider $L\to X$ a nef line bundle, such that 
$\displaystyle H^0\big(X_y, K_{X_y}+ L|_{X_y}\big)\neq 0$. Then the bundle 
$K_{X/Y}+ L$ is pseudo-effective.

\end{corollary}

\begin{proof} Let $A\to X$ be a very ample line bundle. Then for each positive integer 
$m$ we define the bundle 
$$L_m:= mL+ A;$$
it is ample, hence it can be endowed with a metric $h_m$ with positive definite curvature.
We consider the bundle $mK_{X/Y}+ L_m$; by hypothesis, there exists a section 
$ \displaystyle u\in H^0\big(X_y, K_{X_y}+ L|_{X_y}\big)$, so the bundle 
$$mK_{X/Y}+ L_m|_{X_y}$$
admits a non-trivial section, e.g. $u^{\otimes m}\otimes s_A$ where $s_A$ is a non-zero section of $A$. By Theorem 2.1 the bundle $mK_{X/Y}+ L_m$ is pseudo-effective; as 
$m\to \infty$,  we infer that $K_{X/Y}+ L$ is pseudo-effective.
\end{proof}
\medskip

\noindent We will explain next the relevance of the previous result
in the proof of Theorem 1.4 under the assumption that
$X$ is projective; we first recall a few notions.
\smallskip

Let $X$ be a non-singular manifold such that $-K_X$ is nef, and let 
$$\alpha_X: X\to {\rm Alb}(X) $$
be its Albanese morphism. We assume that $\alpha_X$ is not surjective, and let 
$Y\subset {\rm Alb}(X)$ be the image of $\alpha_X$. 

We denote by $\pi_Y: \wh Y\to Y$ the desingularization of $Y$, and let $p: \wh X\to \wh Y$ 
be the map obtained by resolving the indeterminacy of the rational map
$X\dashrightarrow \wh Y$.
 
\noindent We apply Corollary 2.2
with the following data 
$$X:= \wh X, \quad Y:= \wh Y$$ 
and 
$L:= \pi_X^\star (-K_X)$; here we denote by $\pi_X: \wh X\to X$ 
the modification of $X$, so that we have 
$$\pi_Y\circ p= \alpha_X\circ \pi_X.$$ 
The hypothesis required by Corollary 2.2 are quickly seen to be verified:
indeed, the nefness of the bundle $L$ is due to the fact that 
$-K_X$ is nef, and if we denote by $E$ the effective divisor
such that 
\begin{equation} \label{2}K_{\wh X}= \pi^\star (K_X)+ E\end{equation}
then we see that $\displaystyle K_{\wh X_y}+ L$ is simply equal to 
$E|_{\wh X_y}$. This bundle is clearly effective.

Hence we infer that
\begin{equation} \label{3}K_{\wh X/\wh Y}+ \pi_X^\star (-K_X)\end{equation}
is pseudo-effective. But this bundle equals $E- p^\star (K_{\wh Y})$; let 
$\Lambda$ be a closed positive current in the class corresponding to 
$E- p^\star (K_{\wh Y})$. 
Since the Kodaira dimension of $K_{\wh Y}$ is at least 1 (we refer to \cite{kaw0}
for a justification of this property), we obtain 
two $\bQ$-effective divisors say $W_1\neq W_2$ linearly equivalent 
with $K_{\wh Y}$. 
\smallskip

As a conclusion, we obtain two different closed positive currents 
belonging to the class of 
the exceptional divisor $E$, namely $\Lambda+ p^\star (W_j)$ for $j= 1, 2$. 
This gives a contradiction, since any closed positive current linearly equivalent to
$E$ must be $\pi_X$-contractible, so its support is contained in the support of 
$E$. The existence of two closed positive currents having the support contained in $E$
whose cohomology classes coincide shows that one of the irreducible 
components of the support of $E$ must be equal to 
a linear combination of the other components in $H^{1, 1}(X, \bR)$. This is of course absurd.
\qed

\medskip

\begin{rem}\rm 
The proof above shows that Theorem 1.4 still holds if $X$ is projective, and if
we replace the hypothesis \emph{$-K_X$ nef}
with the hypothesis \emph{$-K_X$ pseudo-effective, and the multiplier ideal sheaf associated to some of its positively curved metrics is equal to the structural sheaf }. The arguments are absolutely similar.

\end{rem}

\medskip

\begin{rem}\rm Let $X$ be a Fano manifold, and let $p:X\to Y$ be a submersion onto a non-singular manifold  $Y$. Then it follows that $Y$ is Fano as well (see \cite{KMM}, \cite{fujinoGong}). 
This result can be obtained \emph{via} the following elegant argument,
very recently found and explained to us by 
S. Boucksom. By the results e.g. in \cite{bo3}, the direct image of the bundle 
$$K_{X/Y}+L$$
is positive provided that $L$ is an ample line bundle. We take $L= -K_X$ and we are done. A similar idea, $\varepsilon$-close to our arguments in this section
can be found in the article \cite{demailly3} by J.-P. Demailly,
Th. Peternell and M.Schneider (cf. the proof of their Theorem 2.4).

\end{rem}
\bigskip

\section{Twisted Kahler-Einstein metrics and their variation}

\medskip

\noindent In this paragraph we are going to prove Theorem 1.1. We start (cf. section 3.1)
with a few notations/remarks concerning the metrics induced on $K_{X/Y}$ by a $(1,1)$-form on $X$ which is positive definite along the fibers of $p: X\to Y$. The next section 3.2 is the main part of our paper: we show that the fiberwise twisted K\"ahler-Einstein metric 
(which exists and it is unique, thanks to the fact that the class
$$c_1\big(K_{X_y}\big)+ \{\beta\}|_{X_y}$$
is K\"ahler, for each $y\in Y_0$) endow the bundle $\displaystyle K_{X/Y}|_{X_0}$ with a metric whose curvature is strictly greater than $-\beta$. As we have already mentioned, at this point we will adapt to our setting 
the computations in \cite{schumi1}, \cite{siu1}. In the last subsection of this paragraph, we show that this
metric extends across the singularities of the map $p$. 
\medskip

\subsection{Metrics for the relative canonical bundle of a fibration}

\medskip

\noindent Let $p: X\to Y$ be a surjective map; here $X$ and $Y$ are 
assumed to be compact K\"ahler manifolds. 
We are using the notations/conventions in the introduction, so that the restriction 
$p: X_0\to Y\setminus W$ becomes a surjective, smooth, 
proper map between two complex manifolds. 
\smallskip

We 
consider a smooth (1,1)-form $\rho$ on $X$, whose restriction to the fibers of $p$
is positive definite. Then $\rho$ induces a metric on the bundle $K_{X/Y}$
as follows. 

Let $x\in X$ be a point, and let $U$ be a coordinate set of $X$ centered at 
$x$. We denote by $z^1,..., z^{n+d}$ a coordinate system on $U$, and we equally introduce 
$t^1,..., t^d$ 
coordinates near the point $y= p(x)$. This data induces a trivialization 
of the relative canonical bundle, with respect to which the weight of the 
metric we want to introduce is given by the function $\Psi_U$, defined by the equality
\begin{equation} \label{3}\rho^n \wedge \prod_{j= 1}^d \sqrt{-1}p^\star (dt^j\wedge dt^{\ol j})= e^{\Psi_U} \prod_{i= 1}^{n+d} 
\sqrt{-1}dz^i\wedge dz^{\ol i}.\end{equation}

\noindent Here the dimension of $X$ is assumed to be $n+d$, and the dimension of $Y$
equals $d$.
The functions $(\Psi_U)$ glue together as weights of a globally defined 
metric denoted by $\displaystyle h_{X/Y}^\rho$ on the relative canonical bundle; the corresponding curvature form is 
simply $\sqrt {-1}\ddbar \Psi_U$. 

We remark that we can very well define the function 
$\Psi_U$ even if the point $x$ projects into a singular value of 
$p$; however, the resulting metric $\displaystyle h_{X/Y}^\rho$ 
will be identically $\infty$ 
along the zero set of the Jacobian of the map $p$ (in other words, the weight $\Psi_U$
acquires a log pole). Thus,
the metric $\displaystyle h_{X/Y}^\rho$ will be singular in general, 
even if to start with we are using
a non-singular metric $\rho$. 

A simple example is provided by the map $p: \wh X\to X$, the blow-up of 
a manifold $X$ along a subset $W$. As one can see right away from the formula \eqref{3},
the metric $\displaystyle h_{X/Y}^\rho$
is equal to the singular metric associated to the exceptional divisor
(so it is independent of $\rho$).

\smallskip

\noindent In the next section, we will evaluate the positivity of the curvature 
of the metric $\displaystyle h_{X/Y}^\rho$; for this purpose, we have to 
find a lower bound for the quantity
\begin{equation} \label{4} \sum_{\alpha, \beta}\frac{\partial^2\Psi_U}{\partial z^\alpha\partial z^{\ol \beta}}(x)v^\alpha v^{\ol \beta}\end{equation}
where 
$$v:= \sum_\alpha v^\alpha\frac{\partial}{\partial z^\alpha}$$
is a tangent vector at $X$ in $x$. In order to obtain a lower bound for the quantity 
in (2), it is enough to consider a well-chosen restriction of our initial map
$p$, namely
$$\wt p: X_{\bD}\to \bD$$ 
where $\bD\subset Y$ is a disk containing the point $p(x)$, such that the vector $v$ 
belongs to the tangent space at $x$ to the 
complex manifold $\displaystyle X_{\bD}:= ~p^{-1}(\bD)$. Such a choice is 
clearly possible, and we formulate next our conclusion as follows.
\smallskip

\begin{rem} Let $\gamma$ be a real (1,1)-form on $X$. We assume that for each 
generic enough disk $\bD\subset Y$ so that the analytic subset 
$\displaystyle X_{\bD}:= ~p^{-1}(\bD)$ of $X$ is non-singular we have
$$\Theta_{h_{X_{\bD}/\bD}^\rho}(K_{X_{\bD}/\bD})\geq \gamma|_{X_{\bD}}.$$
 Then we have 
$$\Theta_{h_{{X}/Y}^\rho}(K_{X/Y})\geq \gamma$$ 
 on $X$.
 \end{rem}
 \smallskip
 
 \noindent This is absolutely clear, given the formula (4). Therefore, we will restrict our attention to families over 1-dimensional bases, as long as we are only interested in 
 the curvature properties of the metric $h_{{X}/Y}^\rho$.
 
 \medskip
 
 \noindent After these general considerations, we show here that in the context of Theorem 1.1
 a very special $(1,1)$-form $\rho$ can be obtained as follows.
  
 Let $\omega$ be a K\"ahler form on $X$; as explained at the beginning of the 
 current section, we can construct a metric $h_{X/Y}^\omega$ on the bundle $K_{X/Y}$
 induced by it. We recall that we denote by $\beta$ a semi-positive definite
 form on $X$ given by hypothesis of Theorem 1.1. 
 
 Moreover, we know that for each $y\in Y\setminus W$, the cohomology class
\begin{equation} \label{6}\{\Theta_{h_{X/Y}^\omega}(K_{X/Y})+ \beta\}|_{X_y}\end{equation}
is K\"ahler, again by hypothesis of 1.1. 
Therefore, by the main result of S.-T. Yau in \cite{yau} we have.

\begin{theorem}[\cite{yau}]
There exists a unique function
$\varphi_y\in {\mathcal C}^\infty (X_y)$ such that 
\begin{equation} \label{+}
\Theta_{h_{X/Y}^\omega}(K_{X/Y})+ \beta|_{X_y}+ \sqrt{-1}\ddbar \varphi_y> 0,
\end{equation}
and such that $\varphi_y$ is the solution of the next Monge-Amp\`ere equation
\begin{equation} \label{7}(\Theta_{h_{X/Y}^\omega}(K_{X/Y})+ \beta|_{X_y}+ \sqrt{-1} \ddbar \varphi_y)^n=
e^{\varphi_y}\omega^n.\end{equation}

\end{theorem}
\medskip

\noindent We stress on the fact that the differential form
\begin{equation} \label{8}\Theta_{h_{X/Y}^\omega}(K_{X/Y})+ \beta|_{X_y}\end{equation}
is not necessarily positive definite, but still the equation \eqref{7} admits a solution,
since the cohomology class corresponding to it \eqref{6} is 
K\"ahler. Hence the 
function $\varphi_y$ can be seen to be equal to the sum of two functions: a potential whose Hessian added to \eqref{8} makes it positive definite, and the solution of the Monge-Amp\`ere given by the main theorem in \cite{yau}. The potential we have to add
is by no means unique, but it is the case for the resulting function $\varphi_y$ 
(as one can see thanks to the usual arguments in \cite{yau}). 

Also, an important fact is that the function 
\begin{equation} \label{9}\varphi(x):= \varphi_y(x)\in {\mathcal C}^\infty (X_0),\end{equation}
where $y= p(x)$ is smooth.
That is to say, the function obtained by piecing together the solutions $\varphi_y$ 
of the \eqref{7}
is smooth on $X_0$; this is a standard consequence of the usual estimates for the Monge-Amp\`ere operator.

We consider next the (1,1)-form 
\begin{equation} \label{10}\rho:= \Theta_{h_{X/Y}^\omega}(K_{X/Y})+ \beta+ \sqrt{-1} \ddbar \varphi\end{equation}
on the manifold $X_0$. A first remark is that $\rho$ is definite positive when restricted to
$X_y$, for any $y\in Y_0$: this is contained in Yau's result \cite{yau}.
Thus, we can define a metric $\displaystyle h_{X/Y}^\rho$ on the bundle $\displaystyle K_{X/Y}|_{X_0}$; given the equation \eqref{7}, its curvature is rapidly computed as follows.

\begin{equation} \label{11}\Theta_{h_{X/Y}^\rho}(K_{X/Y})= \rho-\beta
\end{equation}
(we are using the relations \eqref{7} and \eqref{3} in order to derive this). 
\medskip

\noindent There are two main steps in the proof of Theorem 1.1, namely.

\begin{enumerate}

\item[\rm (i)] Show that the form $\rho$ is definite positive on $X_0$; this will imply the
positivity of the form $\displaystyle \Theta_{h_{X/Y}^\rho}(K_{X/Y})+ \beta$
on $X_0$.
\smallskip

\item[\rm (ii)] Show that the metric $h_{X/Y}^\rho$ extends across the set 
$X\setminus X_0$. As soon as this second step is performed, we can infer 
the positivity of $\displaystyle \Theta_{h_{X/Y}^\rho}(K_{X/Y})+ \beta$ as current on $X$.

\end{enumerate}
\smallskip

\noindent These points will be addressed in the next two paragraphs.

 
\subsection{The computation}
\medskip

\noindent As we have already mentioned, in order to analyze the positivity properties of
the for $\rho$ in \eqref{10} it is enough to restrict ourselves to a family over a 1-dimensional base.
Therefore we will assume that the map 
$$p: X\to Y$$
is a proper fibration over a 1-dimensional manifold $Y$; we equally assume that 
$X$ is non-singular, and so it is the generic fiber of $p$. 

Let $x$ be a point of $X$ such that the fiber $\displaystyle X_{y}:= p^{-1}(y)$
is non-singular; here we denote by $y:= p(x)$. Let $t$ be the coordinate on $Y$ centered at 
$y$, and let $(z^1,..., z^n)$ be local coordinates on the manifold $X_{y}$ so that 
the functions $(t, z^1,..., z^n)$ are local coordinates for $X$ at $x$ (we use here the notation $t$ 
for the inverse image of the coordinate on $Y$ via the map $p$, to avoid some notational complications).
\smallskip

With the notations in section 3.1, we write the form $\rho$  in coordinates as follows
\begin{align}
\rho = & \sqrt{-1}g_{t\ol t}dt\wedge d\ol t+ \sqrt{-1}\sum_\alpha g_{\alpha \ol t}dz^\alpha \wedge d\ol t+
\sqrt{-1}\sum_\alpha g_{t \ol \alpha}dt\wedge dz^{\ol \alpha}+ \nonumber \\
+ & \sqrt{-1}\sum_{\alpha, \gamma}g_{\gamma \ol \alpha}dz^{\gamma}\wedge dz^{\ol \alpha}. \nonumber 
\end{align}
We already know that $\rho$ is positive definite when restricted to $X_{y}$, hence it has at least 
$n$ positive eigenvalues. In local writing as above, this implies that the 
matrix $(g_{\gamma \ol \alpha})$ is invertible; we denote the coefficients of 
its inverse by $(g^{\ol \alpha \gamma})$ (with the convention that the $1^{\rm st}$
index is the line index of the associated matrix). 
In order to show that the $n+1^{\rm th}$ eigenvalue (in the ``base direction")
is equally positive,  
we consider the form $\rho^{n+ 1}$ on $X$. As it is well-known, we have
\begin{equation} \label{12}\rho^{n+ 1}= c(\rho)\rho^{n}\wedge \sqrt{-1}dt\wedge d\ol t 
\end{equation}
where the function $c(\rho)$ defined globally on $X_y$ by the preceding 
equality can be expressed locally near $x$ in coordinates as follows
\begin{equation} \label{13}c(\rho)= g_{t\ol t}- \sum_{\alpha, \gamma}g^{\ol \alpha\gamma}g_{t\ol\alpha}g_{\gamma\ol t}.\end{equation}
Our next goal will be to show that we have $\displaystyle c(\rho)|_{X_y}>0$, as this is equivalent to the fact that 
$\rho$ is positive definite. The method (cf. \cite{schumi1}) is to show that this function 
is the solution of a certain elliptic equation on $X_y$. The computations to follow are straightforward.
\smallskip

\noindent Let 
\begin{equation} \label{14}\Box_y:= -\sum_{i, j}g^{\ol j i}\frac{\partial^2}{\partial z^i\partial z^{\ol j}}\end{equation}
be the Laplace operator (with positive eigenvalues) associated to the metric $\rho|_{X_y}$.  
We will evaluate the expression 
\begin{equation} \label{15}\Box_y c(\rho)\end{equation}
by using the local coordinates fixed above; we can (and will) assume that $(z^1,..., z^n)$
are geodesic for the metric $\displaystyle \rho|_{X_{y}}$ at the point $x_0$. 

We first evaluate the expression
\begin{equation} \label{16}-\sum_{i, j}g^{\ol j i}\frac{\partial^2g_{t\ol t}}{\partial z^i\partial z^{\ol j}}; \end{equation}
a first observation, which will be used many times in what follows is that
\begin{equation} \label{17} \frac{\partial^2g_{t\ol t}}{\partial z^i\partial z^{\ol j}}= \frac{\partial^2g_{i\ol j}}{\partial t\partial \ol t}\end{equation}
since $\rho$ is locally $\ddbar$ exact, given the expression \eqref{10}. For any indexes $(i, j)$ we have 
\begin{equation} \label{18} g^{\ol j i} \frac{\partial^2g_{i\ol j}}{\partial t\partial \ol t}= \frac {\partial}{\partial t}
\Big(g^{\ol j i}\frac{\partial g_{i\ol j}}{\partial \ol t} \Big)- \frac {\partial g^{\ol j i}}{\partial t} \frac {\partial g_{i\ol j}}{\partial \ol t}
\end{equation}
The term $\displaystyle \frac {\partial g^{\ol j i}}{\partial t}$ can be written in terms of the 
$t$-derivative of $g_{i\ol j}$ since we have
\begin{equation} \label{19}\sum_s\frac {\partial g^{\ol s i}}{\partial t}g_{k\ol s}= - \sum_sg^{\ol s i}
\frac {\partial g_{k\ol s}}{\partial t}\end{equation}
which implies that
\begin{equation} \label{20}\sum_{s, k}\frac {\partial g^{\ol s i}}{\partial t}g_{k\ol s}g^{\ol j k}= - \sum_{s, k}g^{\ol s i}g^{\ol j k}
\frac {\partial g_{k\ol s}}{\partial t}\end{equation}
and thus we get
\begin{equation} \label{21}\frac {\partial g^{\ol j i}}{\partial t}= - \sum_{s, k}g^{\ol s i}g^{\ol j k}
\frac {\partial g_{k\ol s}}{\partial t}.\end{equation}
We notice that we have
\begin{equation} \label{22}\sum_{i, j}g^{\ol j i}\frac{\partial g_{i\ol j}}{\partial \ol t} = 
\frac{\partial }{\partial \ol t}\log(g)\end{equation}
where $g:= \det (g_{\alpha \ol \beta})$; in conclusion, we obtain the following identity
\begin{equation} \label{23}
-\sum_{i, j}g^{\ol j i}\frac{\partial^2g_{t\ol t}}{\partial z^i\partial z^{\ol j}}=
- \frac{\partial^2\log(g)}{\partial t\partial {\ol t}}-  
\sum_{s, k, i, j}g^{\ol s i}g^{\ol j k}
\frac {\partial g_{k\ol s}}{\partial t}\frac {\partial g_{i\ol j}}{\partial \ol t}
\end{equation}
In local coordinates, the equation \eqref{7} implies that we have
\begin{equation} \label{24}\frac{\partial ^2\log (g)}{\partial t\partial {\ol t}}= g_{t\ol t}- \beta_{t\ol t}\end{equation}
so we get
\begin{equation} \label{25}
-\sum_{i, j}g^{\ol j i}\frac{\partial^2g_{t\ol t}}{\partial z^i\partial z^{\ol j}}=
- g_{t\ol t}+ \beta_{t\ol t}-  
\sum_{s, k, i, j}g^{\ol s i}g^{\ol j k}
\frac {\partial g_{k\ol s}}{\partial t}\frac {\partial g_{i\ol j}}{\partial \ol t}
\end{equation}

We will detail next the computation for the factor
\begin{equation} \label{26} \sum_{i, j,\alpha, \gamma}g^{\ol j i}\frac{\partial^2}{\partial z^i\partial z^{\ol j}}
\Big(g^{\ol \alpha\gamma}g_{t\ol\alpha}g_{\gamma\ol t}\Big);
\end{equation}
given that the coordinates $(z^\alpha)$ are geodesic, the only terms we have to evaluate 
are the following

\begin{equation} \label{27} I_1:= \sum_{i, j,\alpha, \gamma}g^{\ol j i}g_{t\ol\alpha}g_{\gamma\ol t}\frac{\partial^2g^{\ol \alpha\gamma}}{\partial z^i\partial z^{\ol j}}, \end{equation}

as well as 
\begin{equation} \label{27+} 
I_2:= \sum_{i, j,\alpha, \gamma}g^{\ol j i}g^{\ol \alpha\gamma} g_{t\ol\alpha}\frac{\partial^2g_{\gamma\ol t}}{\partial z^i\partial z^{\ol j}}, \quad I_3:= \sum_{i, j,\alpha, \gamma}g^{\ol j i}g^{\ol \alpha\gamma} 
\frac{\partial g_{\gamma\ol t}}{\partial z^{\ol j}}
\frac{\partial g_{t\ol\alpha}}{\partial z^{i}} 
\end{equation}

together with their conjugates, and also
\begin{equation} \label{28} I_4:= \sum_{i, j,\alpha, \gamma}g^{\ol j i}g^{\ol \alpha\gamma} 
\frac{\partial g_{\gamma\ol t}}{\partial z^{i}}
\frac{\partial g_{t\ol\alpha}}{\partial z^{\ol j}}.
\end{equation}
\medskip

\noindent In order to simplify the term $I_1$, we observe that at $x$ we have
\begin{equation} \label{29}\frac{\partial^2g^{\ol \alpha\gamma}}{\partial z^i\partial z^{\ol j}}
= R^{\gamma \ol \alpha}_{i\ol j}\end{equation}
hence we get
\begin{equation} \label{30}I_1= \sum_{i, j,\alpha, \gamma}g^{\ol j i}g_{t\ol\alpha}g_{\gamma\ol t}R^{\gamma \ol \alpha}_{i\ol j}.\end{equation}
We observe that $\displaystyle \sum_{i, j}g^{\ol j i}R^{\gamma \ol \alpha}_{i\ol j}= \sum_{p, q}{\rm Ricci}_{p\ol q}g^{\ol q \gamma}
g^{\ol \alpha p}$ where we denote by $({\rm Ricci}_{p\ol q})$ the coefficients of the Ricci curvature of the 
metric $\rho$. By using the equation \eqref{7} we infer that 
\begin{align}
I_1= & -\sum_{p, q,\alpha, \gamma}g_{t\ol\alpha}g_{\gamma\ol t}g^{\ol q \gamma}
g^{\ol \alpha p}(g_{p\ol q}- \beta_{p\ol q})= \nonumber \\
= & -\sum_{p, q,\gamma}g_{t\ol q}g_{\gamma\ol t}g^{\ol q \gamma}+
\sum_{p, q,\alpha, \gamma}g_{t\ol\alpha}g_{\gamma\ol t}g^{\ol q \gamma}
g^{\ol \alpha p}\beta_{p\ol q} \nonumber
\end{align}
The other terms are evaluated in a similar way; we have

\begin{align}
I_2= & \sum_{i, j,\alpha, \gamma}g^{\ol j i}g^{\ol \alpha\gamma} g_{t\ol\alpha}\frac{\partial^2g_{\gamma\ol t}}{\partial z^i\partial z^{\ol j}}= \sum_{i, j,\alpha, \gamma}g^{\ol j i}g^{\ol \alpha\gamma} g_{t\ol\alpha}\frac{\partial^2g_{i\ol j}}{\partial z^\gamma\partial {\ol t}}= \nonumber \\
= & \sum_{i, j,\alpha, \gamma} g_{t\ol\alpha}g^{\ol \alpha\gamma}\frac{\partial^2\log (g)}{\partial z^\gamma\partial {\ol t}}= 
\nonumber \\
= & \sum_{i, j,\alpha, \gamma} g_{t\ol\alpha}g^{\ol \alpha\gamma}(g_{\gamma \ol t}- \beta_{\gamma \ol t})
= \nonumber \\
= & \sum_{i, j,\alpha, \gamma} g_{t\ol\alpha}g_{\gamma \ol t}g^{\ol \alpha\gamma}
- \sum_{i, j,\alpha, \gamma} g_{t\ol\alpha}g^{\ol \alpha\gamma}\beta_{\gamma \ol t}.
\nonumber 
\end{align}
as well as

\begin{align}
I_3= & \sum_{i, j,\alpha, \gamma}g^{\ol j i}g^{\ol \alpha\gamma} 
\frac{\partial g_{\gamma\ol t}}{\partial z^{\ol j}}
\frac{\partial g_{t\ol\alpha}}{\partial z^{i}}= \sum_{i, j,\alpha, \gamma}g^{\ol j i}g^{\ol \alpha\gamma} 
\frac{\partial g_{\gamma\ol j}}{\partial {\ol t}}
\frac{\partial g_{i\ol\alpha}}{\partial t}.
\nonumber 
\end{align}

\smallskip

\noindent We want to have an intrinsic interpretation of the factor 
$I_4$, so we consider next the vector field 
\begin{equation} \label{32}v:= \frac{\partial}{\partial t}- \sum_{i, j}g^{\ol j i}g_{t\ol j}\frac{\partial}{\partial z^i};\end{equation}
as it is well-known \cite{siu1}, \cite{schumi1}, $v$ is the horizontal lift of $\displaystyle \frac{\partial}{\partial t}$
with respect to the metric $\rho$. Then we see that we have
\begin{equation} \label{33}I_4= |\dbar v|^2,\end{equation}
and by combining all the equalities above, we infer that we have the compact formula
\begin{equation} \label{34}\Box_y c(\rho)= -c(\rho)+ |\dbar v|^2+ \beta (v, \ol v)
\end{equation}
The Ricci curvature of the metric $\displaystyle \rho|_{X_y}$is bounded from below by -1, by the 
equation \eqref{7}; hence precisely as in \cite{schumi1}, we infer that we have 
\begin{equation} \label{35}\inf_{X_y}c(\rho)\geq C\int_{X_y}(|\dbar v|^2+ |v|^2_\beta)dV_\rho\end{equation}
where $C$ only depends on the diameter of the fiber $(X_y, \rho)$. 
Hence, the form $\rho$ is positively defined in the base directions as well. 
\smallskip

In conclusion, 
the fiberwise twisted K\"ahler-Einstein metric $\rho$ defines a K\"ahler metric on $X_0$, the 
pre-image of the set $Y\setminus W$. Given the equation \eqref{11}, this 
means that the curvature of the metric $h^\rho_{X/Y}|_{X_0}$ is bounded by 
$-\beta$.  In the next section, we will show that this
metric extends across the singularities of $p$, and therefore the proof of theorem 1.1 will be complete. \qed
\medskip

\subsection{Extension across the singularities}

\smallskip

\noindent We will use the same notations as in the previous section. Given the point 
$x_0\in X_y$, we will derive an upper bound of the metric induced by 
$\rho$ on the relative canonical bundle. 

Let $\Omega$ be an open coordinate set in $X$ centered at $x_0$. We consider 
$\Omega_y:= \Omega\cap X_y$; we denote by $\psi_\beta$ a local potential of 
the K\"ahler metric $\beta$ on $\Omega$. We recall that we denote by 
$\Psi_\Omega$ the local weight of the metric on $K_{X/Y}$ induced by the 
metric $\omega$ (so implicitly we assume that we have fixed a	coordinate system
$\displaystyle (z^\alpha)_{\alpha= 1,..., n+d}$ on $\Omega$ and $\displaystyle (t^\gamma)_{\gamma= 1,..., d}$ near $y:= p(x_0))$. The function to be bounded 
from above is 
\begin{equation} \label{36}\tau_y:= \Psi_\Omega+ \psi_\beta+ \varphi|_{\Omega_y}\end{equation}

At first glance this may look very simple, since we have 
$\displaystyle \varphi |_{\Omega_y}= \varphi_y$, the solution of \eqref{7}, and thus
the function \eqref{36} is psh on $\Omega_y$. But the bound one can obtain from the 
meanvalue inequality are not good enough for our purposes, since they depend on the 
geometry of the manifold $X_y$: this is precisely what we want to avoid, as $y$ approaches the singular values of $p$.

The idea is first to approximate the function \eqref{36} with $\log$ of absolute values of holomorphic functions; then we show that the holomorphic functions involved in this process admit an extension to $\Omega$, where the use of Cauchy inequalities is 
``legitimate", since the manifold $X$ is non-singular and compact.

\noindent We recall next the following approximation 
result, cf. \cite{demailly4}.
\medskip

\begin{theorem}\cite{demailly4} Let ${\mathcal H}_y^{(m)}$ be the Hilbert space defined as follows
\begin{equation} \label{37}{\mathcal H}_y^{(m)}:= \big\{f\in {\mathcal O}(\Omega_y)\hbox{ such that }  \Vert f\Vert _y^2 = \int_{\Omega_y}|f|^2e^{-m\tau_y- \Vert z\Vert ^2}(dd^c\tau_y)^n< \infty\big\}.\end{equation}
Then we have 
\begin{equation} \label{38}\tau_y(x)= \lim_{m\to \infty}\sup_{f\in {\mathcal H}_y^{(m)}, 
\Vert f\Vert _y^2\leq 1}\frac{1}{m}\log |f(x)|^2\end{equation}
for every $x\in \Omega_y$.

\end{theorem}

\medskip

\noindent In the statement above, the fact that the manifold $\Omega_y$ is Stein 
is of course crucial, since without this hypothesis we cannot approximate $\tau_y$
by using global holomorphic functions. The importance of the volume element 
$(dd^c\tau_y)^n$ will become clear in a moment.

Let $f\in {\mathcal H}_y^{(m)}$ be a holomorphic function, such that 
$\Vert f\Vert _y^2\leq 1$. By H\"older inequality we have
\begin{align}
\int_{\Omega_y}|f|^{2/m}e^{-\tau_y}(dd^c\tau_y)^n\leq &
\Big(\int_{\Omega_y}(dd^c\tau_y)^n\Big)^{\frac{m-1}{m}}\nonumber \\
\leq & C \nonumber
\end{align}
where $C$ can be taken to be the maximum between 1 and the 
volume of the fiber $X_y$ with respect to the K\"ahler class
$\displaystyle c_1(K_{X_y})+ \{\beta\}$; hence, it is a constant independent of 
$m$ and $y$. 

We use now again the equation \eqref{7}: in local coordinates, it can we written as
\begin{equation} \label{39}(dd^c\tau_y)^n= e^{\tau_y- \varphi_\beta}\Big|\frac{dz}{dt}\Big|^2\end{equation}
where the notations are (hopefully...) self-explanatory.
Then the estimate above implies
\begin{equation} \label{40}\int_{\Omega_y}|f|^{2/m}e^{-\varphi_\beta}\Big|\frac{dz}{dt}\Big|^2\leq C\end{equation}
\smallskip

We now invoke the $L^{2/m}$ version of the Ohsawa-Takegoshi theorem obtained in 
\cite{bp3}: it implies the existence of a holomorphic function $F$, such that:

\begin{enumerate} 

\item[(a)] The restriction of $F$ to $\Omega_y$ is equal to $f$.
\smallskip

\item[(b)] There exists a numerical constant $C_0> 0$ 
\emph{independent of $m$} such that
\begin{equation} \label{41}\int_{\Omega}|F|^{2/m}e^{-\varphi_\beta}\big|dz\big|^2\leq C_0
\int_{\Omega_y}|f|^{2/m}e^{-\varphi_\beta}\Big|\frac{dz}{dt}\Big|^2\end{equation}

\end{enumerate}

\smallskip

\noindent In particular, this implies that the value $|F(x_0)|^{2/m}$ is bounded from above by a constant which is independent of $y$ and on $m$. Hence the weight function
$\tau_y$ have the same property (by the restriction statement (a) above), and this 
implies that the metric $h^\rho_{X/Y}|_{X_0}$
\emph{extends} as a singular metric for the relative canonical bundle of the fibration $p$; moreover, its curvature current is greater than $-\beta$. Theorem 1.1 is therefore completely proved. \qed
\medskip

\begin{rem} {\rm As we have already mentioned, in the ``linear" context, the positivity properties of the relative adjoint bundles of type
$K_{X/Y}+ L$ is established in a very explicit way, by showing that the fiberwise Bergman kernel has
a psh variation. The only assumptions 
which are needed to obtain a non-trivial positively curved metric is the positivity of $L$, together with the
existence of an $L^2$ section of $\displaystyle K_{X_y}+ L|_{X_y}$.

In order to compare this theory with our previous considerations, let $p: X\to Y$ be a map 
such that $K_{X_y}$ is ample, for some generic $y\in Y$. In the article 
\cite{emanation1}, H. Tsuji shows that his method of iterating the Bergman kernels can be used to construct inductively a metric on the bundle $mK_{X/Y}+ A$, for any $m\geq 1$ (here we denote by 
$A$ some positive enough line bundle). Moreover, he shows that the metric obtained 
on $K_{X/Y}$ by a limiting process is \emph{precisely} the fiberwise K\"ahler-Einstein metric considered in G. Schumacher paper \cite{schumi1}. As we have seen, the method in \cite{schumi1} has the advantage that 
it offers a \emph{lower bound} for the curvature of the metric constructed on the relative adjoint class, but on the other hand, it should be further extended e.g. to encompass the case where the adjoint class has base points when restricted to fibers. 
}
\end{rem}

\bigskip

\section{Further corollaries, consequences and comments}

\medskip

\subsection{Proof of Corollary 1.2}

\smallskip

\noindent Let $p: \mathcal X\to \bD$
be a K\"ahler family. We denote by $\beta$ a K\"ahler metric on $\mathcal X$. 
For each $\varepsilon> 0$ and for each $t\in \bD$, the class
\begin{equation} \label{42}c_1(K_{{\mathcal X}_t})+ \varepsilon\{\beta\}\end{equation}
is K\"ahler, since by hypothesis the canonical bundle $\displaystyle K_{{\mathcal X}_t}$
is nef. The family $p$ is assumed to be non-singular, hence by the results we have obtained in the section 3.2, the class 
$$c_1(K_{{\mathcal X}/\bD})+ \varepsilon\{\beta\}$$
is K\"ahler. This means that $\displaystyle K_{{\mathcal X}/\bD}$ is nef, so Corollary 1.2 is proved. 
\qed

\medskip

\subsection{Proof of Theorem 1.4}

\smallskip

\noindent Let $X$ be a compact K\"ahler manifold, such that 
$-K_X$ is nef in the sense of the definition 1.3 in the introduction.
We denote by $\alpha_X: X\to {\rm Alb}(X)$ the Albanese morphism of $X$. 

As we have seen in the section 1, a successful approach 
towards the subjectivity of 
the 
map $\alpha_X$ 
in the projective case 
is using in an essential manner the positivity 
properties of the relative canonical bundle associated to the desingularization of 
$\alpha_X$. In the general case we will follow basically the same line of arguments, except that the K\"ahler version of the results 2.1, 2.1 is less general.
\medskip 

\noindent 
Indeed, as a consequence of Theorem 1.1 we infer the next statement.

\begin{corollary} Let $p: X\to Y$ be a surjective map between two
compact K\"ahler manifolds, which are assumed to be non-singular. Let $L\to X$
be a nef line bundle, such that 
the adjoint system $\displaystyle K_{X_y}+ L|_{X_y}$ is equally nef for any 
$y\in Y$ generic. Then the bundle $K_{X/Y}+ L$ is pseudo-effective.
\end{corollary}

\medskip
We recall that in the projective case, instead of the nefness of the adjoint bundle
$\displaystyle K_{X_y}+ L|_{X_y}$ we have assumed that this bundle has a 
non-trivial section. Also, even if in the statement of the preceding corollary 
there is no \emph{transcendental} class involved, its proof is using the full force of
Theorem 1.1: the class $\{\beta\}:= c_1(L)+ \varepsilon{\omega}$ is K\"ahler, for every
positive real $\varepsilon$. This being said, the Corollary 4.1 is a direct consequence of 
Theorem 1.1. \qed
\smallskip

In order to be able to use Corollary 4.1, we first consider a desingularization $\wh Y$ of the 
$\alpha_X(X)$. We denote by $\ol X$ the 
fibered product of $X$ with $\wh Y$ over the base $Y$. This variety 
may be singular, but its singular loci projects into an analytic set 
of $X$ whose co-dimension is at least 2. This is seen e.g. by
considering the rational map $X\dashrightarrow \wh Y$ obtained 
by composing the inverse of $\pi_Y$ with the 
Albanese map $\alpha_X$: this map is defined outside a set of co-dimension
at least 2, and the set $\ol X$ is non-singular at each point of the 
pre-image of this set. Next we invoke the desingularisation result of H. Hironaka
and infer the existence of a map $\wh X\to \ol X$
which is an isomorphism outside the singular set of $\ol X$. 
We note that this procedure does not use the fact that
$X$ is projective. Let $\pi_X: \wh X\to X$ be the map obtained by 
composing the desingularization of $\ol X$ with the natural 
map $\ol X\to X$;
the manifolds/maps constructed above have the next 
important properties

\medskip  

\begin{enumerate}

\item[(i)] \emph{The generic fiber of the map $p: \wh X\to \wh Y$
is disjoint from the support of the exceptional divisor associated to 
the map $\pi_X$ defined by the relation
$$K_{\wh X}= \pi_X^\star(K_X)+ E.$$ 
}

\smallskip

\item[(ii)] \emph{The divisor $E$ is $\pi_X$--contractible})



\end{enumerate}
\medskip

\noindent The rest of the proof of Theorem 1.4 is identical to the arguments invoked in the projective case; the hypothesis of Corollary 4.1 are verified, by the properties (i) and (ii) above.
\bigskip

\begin{rem}{\rm
In a forthcoming paper, we will investigate the singular version of Theorem 1.1; the precise statement can be easily guessed from the projective case
(\cite{janos2}, \cite{eckart2}), as follows. Let $\{\beta\}$ be a 
$(1,1)$-class on $X$, admitting a representative $\Theta= \beta+ \sqrt{-1}\ddbar f$
which is a closed positive current, such that 
$$\int_Xe^{-f}dV< \infty$$
(i.e. $(X, \Theta)$ is the analogue of a klt pair in algebraic geometry).  
If the class 
\begin{equation}\label{54}c_1(K_{X_y})+ \{\beta\}|_{X_y}\end{equation}
contains a	 K\"ahler current for each $y\in Y$ generic, then 
the class $\displaystyle c_1(K_{X/Y})+ \{\beta\}$ should contain a K\"ahler current.
In order to adapt the argument used in this note for the proof of such a result, it seems to us that there are serious technical 
difficulties to overcome. Nevertheless, 
if $\Theta$ is allowed to be singular only along a SNC divisor, and if the class
\eqref{54} is K\"ahler, then a slight generalization of the results in \cite{cgp}, \cite{jmr} 
are enough to conclude. However, it is highly desirable to prove the result in singular context, as the article by J. Koll\'ar \cite{janos2} shows it: one should  
allow base points for the positive representatives of the class
\eqref{54}. We refer to the work of O. Fujino \cite{fujino} for new results and an overview of related topics from the 
algebraic geometry point of view.
}
\end{rem}



\end{document}